\documentclass[twoside,reqno]{amsart}

\usepackage{latexsym}

\newcommand{\qdn}{\hspace*{-1.5mm}}
\newcommand{\qqdn}{\hspace*{-2.5mm}}
\newcommand{\xqdn}{\hspace*{-5.0mm}}
\newcommand{\xxqdn}{\hspace*{-10mm}}

\newcommand{\ffnk}[4]{\left[\qdn\ba{#1}#3\\#4\ea{\!\Big|\:#2}\right]}

\newcommand{\binm}{\binom}

\newcommand{\nnm}{\nonumber}
\newcommand{\be}{\begin{equation}}
\newcommand{\ee}{\end{equation}}
\newcommand{\ba}{\begin{array}}
\newcommand{\ea}{\end{array}}
\newcommand{\bmn}{\begin{eqnarray}}
\newcommand{\emn}{\end{eqnarray}}
\newcommand{\bnm}{\begin{eqnarray*}}
\newcommand{\enm}{\end{eqnarray*}}
\newcommand{\bln}{\begin{subequations}}
\newcommand{\eln}{\end{subequations}}

\newtheorem{thm}{Theorem}
\newtheorem{lemm}[thm]{Lemma}

\newtheorem{prop}[thm]{Proposition}

\newtheorem{entry}{Entry}

\newcommand{\bbtm}[4]{\bibitem{kn:#1}{#2,}~{#3,}~{#4.}}
\newcommand{\cito}[1]{\cite{kn:#1}}
\newcommand{\citu}[2]{\cite[#2]{kn:#1}}

\begin{document} 
{
\title{Summation formulas involving generalized\\ harmonic numbers}
\author{$^{a,b}$Chuanan Wei$^*$, $^c$Xiaoxia Wang}


\dedicatory{
$^A$Department of Mathematics\\
  Shanghai Normal University, Shanghai 200234, China\\
$^B$Department of Information Technology\\
  Hainan Medical College, Haikou 571199, China\\
$^C$Department of Mathematics\\
  Shanghai University, Shanghai 200444, China}

\thanks{\emph{The Correspondence author$^*$}. \emph{Email addresses}: weichuanan78@163.com (C. Wei),
 xwang913@126. com (X. Wang)}


\begin{abstract}
In terms of the derivative operator and three hypergeometric series
identities, several interesting summation formulas involving
generalized harmonic numbers are established.
\end{abstract}

\maketitle\thispagestyle{empty}
\markboth{C. Wei, X. Wang}
         {Generalized harmonic numbers}

\section{Introduction}
For a complex number $x$, define the shifted-factorial to be
\[(x)_{0}=0 \quad\text{and}\quad (x)_{n}
=x(x+1)\cdots(x+n-1)\quad\text{when}\quad n\in\mathbb{N}.\]
 Following Bailey~\cito{bailey}, define the hypergeometric series by
\[_{1+r}F_s\ffnk{cccc}{z}{a_0,&a_1,&\cdots,&a_r}{&b_1,&\cdots,&b_s}
 \:=\:\sum_{k=0}^\infty\frac{(a_0)_k(a_1)_k\cdots(a_r)_k}{(1)_k(b_1)_k\cdots(b_s)_k}z^k,\]
where $\{a_{i}\}_{i\geq0}$ and $\{b_{j}\}_{j\geq1}$ are complex
parameters such that no zero factors appear in the denominators of
the summand on the right hand side. Then Dougall's $_5F_4$-series
identity (cf. \citu{andrews-r}{p. 71}) can be stated as
 \bmn\label{dougall}
 &&_5F_4\ffnk{cccc}{1}{a,1+\frac{a}{2},b,c,d}{\frac{a}{2},1+a-b,1+a-c,1+a-d}
  \nnm\\
&&\:\:=\:\frac{\Gamma(1+a-b)\Gamma(1+a-c)\Gamma(1+a-d)\Gamma(1+a-b-c-d)}
{\Gamma(1+a)\Gamma(1+a-b-c)\Gamma(1+a-b-d)\Gamma(1+a-c-d)},
 \emn
where the parameters satisfy $Re(1+a-b-c-d)>0$ and $\Gamma(x)$ is
the well-known gamma function
\[\Gamma(x)=\int_{0}^{\infty}t^{x-1}e^{-t}dt\:\:\text{with}\:\:Re(x)>0.\]
When $d=a/2$, it reduces to Dixon's $_3F_2
     $-series identity(cf. \citu{andrews-r}{p. 72}):
 \bmn\label{dixon}
 &&\xxqdn\qqdn_3F_2\ffnk{cccc}{1}{a,b,c}{1+a-b,1+a-c}
  \nnm\\
&&\xxqdn\qqdn\:\:=\:\frac{\Gamma(1+\frac{a}{2})\Gamma(1+a-b)\Gamma(1+a-c)\Gamma(1+\frac{a}{2}-b-c)}
{\Gamma(1+a)\Gamma(1+\frac{a}{2}-b)\Gamma(1+\frac{a}{2}-c)\Gamma(1+a-b-c)}
 \emn
provided that $Re(1+\frac{a}{2}-b-c)>0$. A Dixon-like identity that
will appear in Section 3 is \bmn\label{dixon-like}
 &&\xxqdn\qqdn_3F_2\ffnk{cccc}{1}{a,b,c}{1+a-b,a-c}
  \nnm\\\nnm
&&\xxqdn\qqdn\:\:=\:\frac{1}{2^{1+c}}\frac{\Gamma(1+a-b)\Gamma(\frac{1+a}{2}-b-c)\Gamma(\frac{a-c}{2})\Gamma(\frac{1+a-c}{2})}
{\Gamma(1+a-b-c)\Gamma(\frac{a}{2})\Gamma(\frac{1+a}{2}-b)\Gamma(\frac{1+a}{2}-c)}
\\
&&\xxqdn\qqdn\:\:+\:\,\frac{1}{2^{1+c}}\frac{\Gamma(1+a-b)\Gamma(\frac{2+a}{2}-b-c)\Gamma(\frac{a-c}{2})\Gamma(\frac{1+a-c}{2})}
{\Gamma(1+a-b-c)\Gamma(\frac{1+a}{2})\Gamma(\frac{2+a}{2}-b)\Gamma(\frac{a}{2}-c)}
 \emn
provided that $Re(\frac{1+a}{2}-b-c)>0$.

For a complex number $x$ and a positive integer $\ell$, define
generalized harmonic numbers of $\ell$-order to be
\[H_{0}^{\langle \ell\rangle}(x)=0
\quad\text{and}\quad
 H_{n}^{\langle\ell\rangle}(x)=\sum_{k=1}^n\frac{1}{(x+k)^{\ell}}
 \quad\text{with}\quad n\in\mathbb{N}.\]
When $x=0$, they become harmonic numbers of $\ell$-order
\[H_{0}^{\langle \ell\rangle}=0
\quad\text{and}\quad
  H_{n}^{\langle \ell\rangle}
  =\sum_{k=1}^n\frac{1}{k^{\ell}} \quad\text{with}\quad n\in\mathbb{N}.\]
  Setting $\ell=1$ in $H_{0}^{\langle \ell\rangle}(x)$ and $H_{n}^{\langle \ell\rangle}(x)$, we obtain
generalized harmonic numbers
\[H_{0}(x)=0
\quad\text{and}\quad H_{n}(x)
  =\sum_{k=1}^n\frac{1}{x+k} \quad\text{with}\quad n\in\mathbb{N}.\]
When $x=0$, they reduce to classical harmonic numbers
\[H_{0}=0\quad\text{and}\quad
H_{n}=\sum_{k=1}^n\frac{1}{k} \quad\text{with}\quad
n\in\mathbb{N}.\]
 For a differentiable function $f(x)$, define the derivative operator
$\mathcal{D}_x$ by
 \bnm
&&\mathcal{D}_xf(x)=\frac{d}{dx}f(x).
 \enm
 Then it is not difficult to show that
$$\mathcal{D}_x\:\binm{x+s}{t}=\binm{x+s}{t}\big\{H_s(x)-H_{s-t}(x)\big\},$$
$$\mathcal{D}_x\:H_{n}^{\langle \ell\rangle}(x) =-\ell
H_{n}^{\langle\ell+1\rangle}(x),$$ where $s,t\in\mathbb{N}_0$ with
$t\leq s$.

As pointed out by Richard Askey (cf. \cito{andrews}), expressing
harmonic numbers in accordance with differentiation of binomial
coefficients can be traced back to Issac Newton.
 In 2003, Paule and Schneider
\cito{paule} computed the family of series:
 \bnm
W_n(\alpha)=\sum_{k=0}^n\binm{n}{k}^{\alpha}\{1+\alpha(n-2k)H_k\}
 \enm
with $\alpha=1,2,3,4,5$ by combining this way with Zeilberger's
algorithm for definite hypergeometric sums. According to the
derivative operator and the hypergeometric form of Andrews'
$q$-series transformation, Krattenthaler and Rivoal
\cito{krattenthaler} deduced general Paule-Schneider type identities
with $\alpha$ being a positive integer.  More results from
differentiation of binomial coefficients can be found in the papers
\cite{kn:sofo-a,kn:wang-b,kn:wei-a,kn:wei-b}. For different ways and
related harmonic number identities, the reader may refer to
\cite{kn:chen,kn:chyzak,kn:graham,kn:kronenburg-a,
kn:kronenburg-b,kn:schneider,kn:sofo-b,kn:wang-a}. It should be
mentioned that Sun \cito{sun} showed recently some congruence
relations concerning harmonic numbers to us.

Inspired by the work just mentioned, we shall establish several
interesting summation formulas involving generalized harmonic
numbers by means of the derivative operator and
\eqref{dougall}-\eqref{dixon-like}. By specifying the parameters,
they can give numerous harmonic number identities. For making the
reader have a taste, we select, above all, the following two ones:
 \bnm
 &&\sum_{k=0}^{2n}\frac{(-1)^k}{\binm{2n}{k}}H_{k}^{\langle2\rangle}
=\frac{1+2n}{2+2n}H_{n}^{\langle2\rangle},\\
 &&\sum_{k=0}^{2n}\frac{(-1)^k}{\binm{2n}{k}}H_{k}^2
=\frac{1+2n}{2+2n}\bigg\{H_{1+2n}^2-\frac{H_{1+2n}}{1+n}
 -H_{1+2n}^{\langle2\rangle}+H_{1+n}^{\langle2\rangle}\bigg\},
 \enm
where the first equation comes from the case $p=0$ of
\eqref{harmonic-bb} and the second equation is exactly Proposition
\ref{prop-a}.
\section{Dixon's identity, reversal techniques and summation formulas \\involving generalized harmonic numbers}
\begin{thm} \label{thm-a}
Let $x$ and $y$ be complex numbers. Then
 \bnm
 &&\xxqdn\sum_{k=0}^{2n}(-1)^k\binm{2n}{k}\frac{\binm{x+k}{k}\binm{y+k}{k}}{\binm{x+2n}{k}\binm{y+2n}{k}}H_{k}(x)\\
&&\xxqdn\:\:=\:\frac{1}{2}\frac{\binm{x+n}{n}\binm{y+n}{n}\binm{1+x+y+2n}{2n}}{\binm{x+2n}{2n}\binm{y+2n}{2n}\binm{1+x+y+n}{n}}
\big\{H_n(x)-H_n(1+x+y)+H_{2n}(1+x+y)\big\}.
 \enm
\end{thm}

\begin{proof}
 Perform the replacements
$a\to-2n$, $b\to 1+x$, $c\to 1+y$ in \eqref{dixon} to get
 \bmn\label{equation-a}
 \xqdn\qqdn\sum_{k=0}^{2n}(-1)^k\binm{2n}{k}\frac{\binm{x+k}{k}\binm{y+k}{k}}{\binm{x+2n}{k}\binm{y+2n}{k}}
=\frac{\binm{x+n}{n}\binm{y+n}{n}\binm{1+x+y+2n}{2n}}{\binm{x+2n}{2n}\binm{y+2n}{2n}\binm{1+x+y+n}{n}}.
 \emn
 Applying the derivative operator $\mathcal{D}_x$ to both sides of
 \eqref{equation-a}, we gain
\bnm
 &&\xxqdn\qqdn\sum_{k=0}^{2n}(-1)^k\binm{2n}{k}\frac{\binm{x+k}{k}\binm{y+k}{k}}{\binm{x+2n}{k}\binm{y+2n}{k}}\{H_{k}(x)+H_{2n-k}(x)\}\\
&&\xxqdn\qqdn\:\:=\:\frac{\binm{x+n}{n}\binm{y+n}{n}\binm{1+x+y+2n}{2n}}{\binm{x+2n}{2n}\binm{y+2n}{2n}\binm{1+x+y+n}{n}}
\big\{H_n(x)-H_n(1+x+y)+H_{2n}(1+x+y)\big\}.
 \enm
In terms of the reversal techniques, it is not difficult to show
that
 \bnm
\qquad\:\:\sum_{k=0}^{2n}(-1)^k\binm{2n}{k}\frac{\binm{x+k}{k}\binm{y+k}{k}}{\binm{x+2n}{k}\binm{y+2n}{k}}H_{k}(x)
=\sum_{k=0}^{2n}(-1)^k\binm{2n}{k}\frac{\binm{x+k}{k}\binm{y+k}{k}}{\binm{x+2n}{k}\binm{y+2n}{k}}H_{2n-k}(x).
\enm
 Therefore, we derive Theorem \ref{thm-a} to complete the proof.
\end{proof}

Taking $x=p$, $y=q$ with $p,q\in \mathbb{N}_0$ in Theorem
\ref{thm-a} and utilizing \eqref{equation-a}, we achieve the
summation formula on harmonic numbers:
 \bmn\label{harmonic-aa}
 &&\sum_{k=0}^{2n}(-1)^k\binm{2n}{k}\frac{\binm{p+k}{k}\binm{q+k}{k}}{\binm{p+2n}{k}\binm{q+2n}{k}}H_{p+k}
 \nnm\\
&&\:\:=\:\frac{1}{2}\frac{\binm{p+n}{n}\binm{q+n}{n}\binm{1+p+q+2n}{2n}}{\binm{p+2n}{2n}\binm{q+2n}{2n}\binm{1+p+q+n}{n}}
\big\{H_{p}+H_{p+n}-H_{1+p+q+n}+H_{1+p+q+2n}\big\}.
 \emn

\begin{thm} \label{thm-b}
Let $x$ be a complex number. Then
 \bnm
 \sum_{k=0}^{2n}(-1)^k\binm{2n}{k}\frac{\binm{x+k}{k}^2}{\binm{x+2n}{k}^2}H_{k}^{\langle2\rangle}(x)
=\frac{1}{2}\frac{\binm{x+n}{n}^2\binm{1+2x+2n}{2n}}{\binm{x+2n}{2n}^2\binm{1+2x+n}{n}}H_{n}^{\langle2\rangle}(x).
 \enm
\end{thm}

\begin{proof}
Applying the derivative operator $\mathcal{D}_y$ to Theorem
\ref{thm-a} and then fixing $y=x$, we
 attain
\bmn\label{equation-b}
 &&\xxqdn\sum_{k=0}^{2n}(-1)^k\binm{2n}{k}\frac{\binm{x+k}{k}^2}{\binm{x+2n}{k}^2}
\{H_{k}^2(x)+H_{k}(x)H_{2n-k}(x)\}=\frac{1}{2}\frac{\binm{x+n}{n}^2\binm{1+2x+2n}{2n}}{\binm{x+2n}{2n}^2\binm{1+2x+n}{n}}
 \nnm\\
&&\xxqdn\:\:\times\:
\big\{H_{n}^{\langle2\rangle}(1+2x)-H_{2n}^{\langle2\rangle}(1+2x)+[H_n(1+2x)-H_{2n}(1+2x)-H_n(x)]^2\big\}.
 \emn
 Applying the derivative operator $\mathcal{D}_x$ to Theorem \ref{thm-a} and then setting $y=x$, we
obtain
 \bnm
 &&\xxqdn\sum_{k=0}^{2n}(-1)^k\binm{2n}{k}\frac{\binm{x+k}{k}^2}{\binm{x+2n}{k}^2}
\{H_{k}^2(x)+H_{k}(x)H_{2n-k}(x)-H_{k}^{\langle2\rangle}(x)\}\\
&&\xxqdn\:\:=\:\frac{1}{2}\frac{\binm{x+n}{n}^2\binm{1+2x+2n}{2n}}{\binm{x+2n}{2n}^2\binm{1+2x+n}{n}}
\big\{H_{n}^{\langle2\rangle}(1+2x)-H_{2n}^{\langle2\rangle}(1+2x)-H_{n}^{\langle2\rangle}(x)
\\&&\qquad\qquad\qquad\quad+[H_n(1+2x)-H_{2n}(1+2x)-H_n(x)]^2\big\}.
 \enm
The difference of \eqref{equation-b} and the last equation produces
Theorem \ref{thm-b}.
\end{proof}

Taking $x=p$ with $p\in \mathbb{N}_0$ in Theorem \ref{thm-b} and
using \eqref{equation-a}, we get the summation formula on harmonic
numbers:
 \bmn\label{harmonic-bb}
\qquad
\sum_{k=0}^{2n}(-1)^k\binm{2n}{k}\frac{\binm{p+k}{k}^2}{\binm{p+2n}{k}^2}H_{p+k}^{\langle2\rangle}
=\frac{1}{2}\frac{\binm{p+n}{n}^2\binm{1+2p+2n}{2n}}{\binm{p+2n}{2n}^2\binm{1+2p+n}{n}}
\big\{H_{p+n}^{\langle2\rangle}+H_{p}^{\langle2\rangle}\big\}.
 \emn

\begin{prop}[Harmonic number identity]\label{prop-a}
 \bnm
 \sum_{k=0}^{2n}\frac{(-1)^k}{\binm{2n}{k}}H_{k}^2
=\frac{1+2n}{2+2n}\bigg\{H_{1+2n}^2-\frac{H_{1+2n}}{1+n}
 -H_{1+2n}^{\langle2\rangle}+H_{1+n}^{\langle2\rangle}\bigg\}.
 \enm
\end{prop}

\begin{proof}
The case $x=0$ of \eqref{equation-b} reads as
 \bnm
\:\qquad\sum_{k=0}^{2n}\frac{(-1)^k}{\binm{2n}{k}}\big\{H_{k}^2+H_kH_{2n-k}\big\}
=\frac{1+2n}{2+2n}\bigg\{\bigg(H_{1+2n}-\frac{1}{1+n}\bigg)^2
 -H_{1+2n}^{\langle2\rangle}+H_{1+n}^{\langle2\rangle}\bigg\}.
 \enm
The case $p=n$ of Wei, Gong and Yan \citu{wei-a}{Corollary 21} is
 \bnm
\sum_{k=0}^{2n}\frac{(-1)^k}{\binm{2n}{k}}H_kH_{2n-k}
=\frac{1+2n}{2(1+n)^2}\bigg\{\frac{1}{1+n}-H_{1+2n}\bigg\}.
 \enm
The difference of the last two equations offers Proposition
\ref{prop-a}.
\end{proof}

\section{Dixon-like identity, bisection method and summation formulas \\involving generalized harmonic numbers}
\begin{lemm}[Dixon-like identity]\label{lemm}
\bnm
 &&\xxqdn\qqdn_3F_2\ffnk{cccc}{1}{a,b,c}{1+a-b,a-c}
  \nnm\\\nnm
&&\xxqdn\qqdn\:\:=\:\frac{1}{2^{1+c}}\frac{\Gamma(1+a-b)\Gamma(\frac{1+a}{2}-b-c)\Gamma(\frac{a-c}{2})\Gamma(\frac{1+a-c}{2})}
{\Gamma(1+a-b-c)\Gamma(\frac{a}{2})\Gamma(\frac{1+a}{2}-b)\Gamma(\frac{1+a}{2}-c)}
\\
&&\xxqdn\qqdn\:\:+\:\,\frac{1}{2^{1+c}}\frac{\Gamma(1+a-b)\Gamma(\frac{2+a}{2}-b-c)\Gamma(\frac{a-c}{2})\Gamma(\frac{1+a-c}{2})}
{\Gamma(1+a-b-c)\Gamma(\frac{1+a}{2})\Gamma(\frac{2+a}{2}-b)\Gamma(\frac{a}{2}-c)}
 \enm
provided that $Re(\frac{1+a}{2}-b-c)>0$.
\end{lemm}

\begin{proof}
Recall Whipple's $_3F_2$-series identity (cf. \citu{andrews-r}{p.
149}):
 \bnm
 \:\:_3F_2\ffnk{cccc}{1}{a,1-a,b}{c,1+2b-c}
=\frac{\pi 2^{1-2b}\Gamma(c)\Gamma(1+2b-c)}
{\Gamma(\frac{a+c}{2})\Gamma(\frac{1+a-c}{2}+b)\Gamma(\frac{1-a+c}{2})\Gamma(\frac{2-a-c}{2}+b)},
 \enm
where $Re(b)>0$. Employ the substitution $a\to 1+a$ in the last
equation to gain
 \bnm
\quad _3F_2\ffnk{cccc}{1}{1+a,-a,b}{c,1+2b-c} =\frac{\pi
2^{1-2b}\Gamma(c)\Gamma(1+2b-c)}
{\Gamma(\frac{1+a+c}{2})\Gamma(\frac{2+a-c}{2}+b)\Gamma(\frac{-a+c}{2})\Gamma(\frac{1-a-c}{2}+b)}.
 \enm
 The linear combination of the last two equations
gives
 \bmn\label{whipple-like}
 _3F_2\ffnk{cccc}{1}{a,-a,b}{c,1+2b-c}&&\xqdn\!=\frac{1}{2}\frac{\pi 2^{1-2b}\Gamma(c)\Gamma(1+2b-c)}
{\Gamma(\frac{a+c}{2})\Gamma(\frac{1+a-c}{2}+b)\Gamma(\frac{1-a+c}{2})\Gamma(\frac{2-a-c}{2}+b)}
\nnm\\&&\xqdn\!+\:\frac{1}{2}\frac{\pi2^{1-2b}\Gamma(c)\Gamma(1+2b-c)}
{\Gamma(\frac{1+a+c}{2})\Gamma(\frac{2+a-c}{2}+b)\Gamma(\frac{-a+c}{2})\Gamma(\frac{1-a-c}{2}+b)}.
 \emn
 According to Kummer's transformation formula (cf.
\citu{andrews-r}{p. 142}):
  \bnm \quad
\!\!_3F_2\ffnk{cccc}{1}{a,b,c}{d,e}
 =\frac{\Gamma(e)\Gamma(d+e-a-b-c)}{\Gamma(e-a)\Gamma(d+e-b-c)}
{_3F_2}\ffnk{cccc}{1}{a,d-b,d-c}{d,d+e-b-c},
 \enm
we achieve
  \bnm
  _3F_2\ffnk{cccc}{1}{c,a,b}{a-c,1+a-b}
 &&\xqdn\!=\frac{\Gamma(1+a-b)\Gamma(1+a-2b-2c)}{\Gamma(1+a-b-c)\Gamma(1+a-2b-c)}\\
 &&\xqdn\!\times\:{_3F_2}\ffnk{cccc}{1}{c,-c,a-b-c}{a-c,1+a-2b-c}.
 \enm
Evaluating the series on the right hand side by
\eqref{whipple-like}, we attain Lemma \ref{lemm} to finish the
proof.
\end{proof}

\begin{thm} \label{thm-c}
Let $x$ be a complex number. Then
 \bnm
 &&\xxqdn\sum_{k=0}^{n}(-1)^k\binm{n}{k}\frac{\binm{2x+k}{k}}{\binm{2x+n+k}{k}}H_{k}(x)\\
&&\xxqdn\:\:=\:4^{n-1}\frac{\binm{n-\frac{1}{2}}{n}}{\binm{2x+2n}{n}}\big\{H_n(x)+H_n-2H_{2n}\big\}
-\frac{4^{n-1}}{n}\frac{\binm{x+n-\frac{1}{2}}{n}}{\binm{2x+2n}{n}\binm{x+n}{n}}.
 \enm
\end{thm}

\begin{proof}
The case $c=-n$ of Lemma \ref{lemm} can be written as
 \bnm
 &&\xxqdn\qqdn_3F_2\ffnk{cccc}{1}{a,b,-n}{1+a-b,a+n}\\
&&\xxqdn\qqdn\:\:=\:2^{2n-1}\frac{(\frac{1+a}{2})_n(\frac{2+a}{2}-b)_n}{(a+n)_n(1+a-b)_n}
+2^{2n-1}\frac{(\frac{a}{2})_n(\frac{1+a}{2}-b)_n}{(a+n)_n(1+a-b)_n}.
 \enm
Replace respectively $a$ and $b$ by $1+x$ and $1+y$ in the last
equation to obtain
 \bmn\label{equation-c}
  &&\xxqdn\sum_{k=0}^{n}(-1)^k\binm{n}{k}\frac{\binm{x+k}{k}\binm{y+k}{k}}{\binm{x+n+k}{k}\binm{x-y+k}{k}}
 \nnm\\&&\xxqdn\:\:=\:2^{2n-1}\frac{\binm{\frac{x}{2}+n}{n}\binm{\frac{x-1}{2}-y+n}{n}}{\binm{x+2n}{n}\binm{x-y+n}{n}}
+2^{2n-1}\frac{\binm{\frac{x-1}{2}+n}{n}\binm{\frac{x-2}{2}-y+n}{n}}{\binm{x+2n}{n}\binm{x-y+n}{n}}.
 \emn
 Applying the derivative operator $\mathcal{D}_y$ to both sides of
 \eqref{equation-c}, we get
 \bmn\label{equation-d}
  &&\xqdn\sum_{k=0}^{n}(-1)^k\binm{n}{k}\frac{\binm{x+k}{k}\binm{y+k}{k}}{\binm{x+n+k}{k}\binm{x-y+k}{k}}
  \big\{H_k(y)+H_k(x-y)\big\}
 \nnm\\\nnm&&\xqdn\:\:=\:2^{2n-1}\frac{\binm{\frac{x}{2}+n}{n}\binm{\frac{x-1}{2}-y+n}{n}}{\binm{x+2n}{n}\binm{x-y+n}{n}}
 \big\{H_n(x-y)-H_n(\tfrac{x-1}{2}-y)\big\}
\\&&\xqdn\:\:+\:\:
2^{2n-1}\frac{\binm{\frac{x-1}{2}+n}{n}\binm{\frac{x-2}{2}-y+n}{n}}{\binm{x+2n}{n}\binm{x-y+n}{n}}
\big\{H_n(x-y)-H_n(\tfrac{x-2}{2}-y)\big\}.
 \emn
In accordance with the relation
 \bnm
\qquad H_n(\tfrac{x-1}{2}-y)=2H_{2n}(x-2y)-H_n(\tfrac{x}{2}-y),
 \enm
\eqref{equation-d} can be reformulated as
 \bnm
  &&\xqdn\sum_{k=0}^{n}(-1)^k\binm{n}{k}\frac{\binm{x+k}{k}\binm{y+k}{k}}{\binm{x+n+k}{k}\binm{x-y+k}{k}}
  \big\{H_k(y)+H_k(x-y)\big\}
 \\&&\xqdn\:\:=\:2^{2n-1}\frac{\binm{\frac{x}{2}+n}{n}\binm{\frac{x-1}{2}-y+n}{n}}{\binm{x+2n}{n}\binm{x-y+n}{n}}
 \big\{H_n(x-y)+H_n(\tfrac{x}{2}-y)-2H_{2n}(x-2y)\big\}
\\&&\xqdn\:\:+\:\:
2^{2n-1}\frac{\binm{\frac{x-1}{2}+n}{n}\binm{\frac{x-2}{2}-y+n}{n}}{\binm{x+2n}{n}\binm{x-y+n}{n}}
\bigg\{H_n(x-y)-H_n(\tfrac{x}{2}-y)+\frac{2}{x-2y+2n}\bigg\}
\\&&\xqdn\:\:-\:\:
2^{2n-1}\frac{\binm{\frac{x-1}{2}+n}{n}\binm{\tfrac{x}{2}-y+n}{n}}{\binm{x+2n}{n}\binm{x-y+n}{n}}\frac{2}{x-2y+2n}.
 \enm
Substitute respectively $x$ and $y$ by $2x$ and $x$ in the last
equation to produce Theorem \ref{thm-c}.
\end{proof}

Fixing $x=p$ with $p\in \mathbb{N}_0$ in Theorem \ref{thm-c} and
exploiting \eqref{equation-c}, we gain the summation formula on
harmonic numbers:
 \bmn\label{harmonic-cc}
 &&\xxqdn\sum_{k=0}^{n}(-1)^k\binm{n}{k}\frac{\binm{2p+k}{k}}{\binm{2p+n+k}{k}}H_{p+k}
  \nnm\\
&&\xxqdn\:\:=\:4^{n-1}\frac{\binm{n-\frac{1}{2}}{n}}{\binm{2p+2n}{n}}\big\{H_{p+n}+H_p+H_n-2H_{2n}\big\}
-\frac{4^{n-1}}{n}\frac{\binm{p+n-\frac{1}{2}}{n}}{\binm{2p+2n}{n}\binm{p+n}{n}}.
 \emn

 \begin{thm} \label{thm-d}
Let $x$ be a complex number. Then
 \bnm
 &&\xxqdn\xqdn\qdn\sum_{k=0}^{n}(-1)^k\binm{n}{k}
 \frac{\binm{\frac{x}{2}+k}{k}\binm{x-\frac{1}{2}+k}{k}}{\binm{\frac{x-1}{2}+k}{k}\binm{x-\frac{1}{2}+n+k}{k}}H_{2k}(x)\\
&&\xxqdn\xqdn\qdn\:\:=\:4^{n-1}\frac{\binm{\frac{x}{2}-\frac{1}{4}+n}{n}\binm{-\frac{3}{4}+n}{n}}{\binm{\frac{x-1}{2}+n}{n}\binm{x-\frac{1}{2}+2n}{n}}
\big\{H_n(\tfrac{x-1}{2})-H_n(-\tfrac{3}{4})\big\}\\
&&\xxqdn\xqdn\qdn\:\:+\:\:4^{n-1}\frac{\binm{\frac{x}{2}-\frac{3}{4}+n}{n}\binm{-\frac{5}{4}+n}{n}}{\binm{\frac{x-1}{2}+n}{n}\binm{x-\frac{1}{2}+2n}{n}}
\big\{H_n(\tfrac{x-1}{2})-H_n(-\tfrac{5}{4})\big\}.
 \enm
\end{thm}

\begin{proof}
Perform the replacements $x\to x-\frac{1}{2}$, $y\to \frac{x}{2}$ in
\eqref{equation-d} to achieve
 \bnm
 &&\xxqdn\sum_{k=0}^{n}(-1)^k\binm{n}{k}
 \frac{\binm{\frac{x}{2}+k}{k}\binm{x-\frac{1}{2}+k}{k}}{\binm{\frac{x-1}{2}+k}{k}\binm{x-\frac{1}{2}+n+k}{k}}
 \big\{H_{k}(\tfrac{x}{2})+H_{k}(\tfrac{x-1}{2})\big\}\\
&&\xxqdn\:\:=\:2^{2n-1}\frac{\binm{\frac{x}{2}-\frac{1}{4}+n}{n}\binm{-\frac{3}{4}+n}{n}}{\binm{\frac{x-1}{2}+n}{n}\binm{x-\frac{1}{2}+2n}{n}}
\big\{H_n(\tfrac{x-1}{2})-H_n(-\tfrac{3}{4})\big\}\\
&&\xxqdn\:\:+\:\:2^{2n-1}\frac{\binm{\frac{x}{2}-\frac{3}{4}+n}{n}\binm{-\frac{5}{4}+n}{n}}{\binm{\frac{x-1}{2}+n}{n}\binm{x-\frac{1}{2}+2n}{n}}
\big\{H_n(\tfrac{x-1}{2})-H_n(-\tfrac{5}{4})\big\}.
 \enm
By means of the relation
 \bmn\label{relation}
H_{k}(\tfrac{x}{2})+H_{k}(\tfrac{x-1}{2})=2H_{2k}(x),
 \emn
the last equation can be expressed as Theorem \ref{thm-d} to
complete the proof.
\end{proof}

Setting $x=p$ with $p\in \mathbb{N}_0$ in Theorem \ref{thm-d} and
utilizing \eqref{equation-c}, we attain the summation formula on
harmonic numbers and generalized harmonic numbers:
 \bmn\label{harmonic-dd}
 &&\xqdn\qdn\sum_{k=0}^{n}(-1)^k\binm{n}{k}
 \frac{\binm{\frac{p}{2}+k}{k}\binm{p-\frac{1}{2}+k}{k}}{\binm{\frac{p-1}{2}+k}{k}\binm{p-\frac{1}{2}+n+k}{k}}H_{p+2k}
 \nnm\\\nnm
&&\xqdn\qdn\:\:=\:4^{n-1}\frac{\binm{\frac{p}{2}-\frac{1}{4}+n}{n}\binm{-\frac{3}{4}+n}{n}}{\binm{\frac{p-1}{2}+n}{n}\binm{p-\frac{1}{2}+2n}{n}}
\big\{2H_p+H_n(\tfrac{p-1}{2})-H_n(-\tfrac{3}{4})\big\}\\
&&\xqdn\qdn\:\:+\:\:4^{n-1}\frac{\binm{\frac{p}{2}-\frac{3}{4}+n}{n}\binm{-\frac{5}{4}+n}{n}}{\binm{\frac{p-1}{2}+n}{n}\binm{p-\frac{1}{2}+2n}{n}}
\big\{2H_p+H_n(\tfrac{p-1}{2})-H_n(-\tfrac{5}{4})\big\}.
 \emn
\section{Dougall's identity, bisection method and summation formulas \\involving generalized harmonic numbers}
\begin{thm} \label{thm-e}
Let $x$ and $y$ be complex numbers. Then
 \bnm
 &&\xxqdn\sum_{k=0}^{n}(-1)^k\binm{n}{k}\frac{\binm{2x+k}{k}\binm{2y+k}{k}}{\binm{2x+n+k}{k}\binm{2x-2y+k}{k}}
 \frac{1+2x+2k}{1+2x+n+k}H_{k}(x)\\
&&\xxqdn\:\:=\:\frac{1}{2}\frac{\binm{2x+n}{n}\binm{x-2y-1+n}{n}}{\binm{x+n}{n}\binm{2x-2y+n}{n}}\big\{H_n(x)-H_{n}(x-2y-1)\big\}.
 \enm
\end{thm}

\begin{proof}
Employ the substitutions $a\to1+x$, $b\to1+y$, $c\to1+z$ in
\eqref{dougall} to obtain
 \bmn\label{equation-e}
  &&\xxqdn\sum_{k=0}^{n}(-1)^k\binm{n}{k}\frac{\binm{x+k}{k}\binm{y+k}{k}\binm{z+k}{k}}
  {\binm{x+n+k}{k}\binm{x-y+k}{k}\binm{x-z+k}{k}}\frac{1+x+2k}{1+x+n+k}=
\frac{\binm{x+n}{n}\binm{x-y-z-1+n}{n}}{\binm{x-y+n}{n}\binm{x-z+n}{n}}.
 \emn
 Applying the derivative operator $\mathcal{D}_z$ to both sides of
 \eqref{equation-e}, we get
 \bmn\label{equation-f}
  &&\xxqdn\sum_{k=0}^{n}(-1)^k\binm{n}{k}\frac{\binm{x+k}{k}\binm{y+k}{k}\binm{z+k}{k}}
  {\binm{x+n+k}{k}\binm{x-y+k}{k}\binm{x-z+k}{k}}\frac{1+x+2k}{1+x+n+k}
  \big\{H_k(z)+H_k(x-z)\big\}
 \nnm\\&&\xxqdn\:\:=\:
 \frac{\binm{x+n}{n}\binm{x-y-z-1+n}{n}}{\binm{x-y+n}{n}\binm{x-z+n}{n}}
 \big\{H_n(x-z)-H_n(x-y-z-1)\big\}.
 \emn
Replace respectively $x$, $y$ and $z$ by $2x$, $2y$ and $x$ in the
last equation to offer Theorem \ref{thm-e}.
\end{proof}

Taking $x=p$, $y=\frac{q}{2}$ with $p,q\in \mathbb{N}_0$ in Theorem
\ref{thm-e} and using \eqref{equation-e}, we gain the summation
formula on harmonic numbers:
 \bmn\label{harmonic-ee}
 &&\xxqdn\xqdn\qdn\sum_{k=0}^{n}(-1)^k\binm{n}{k}\frac{\binm{2p+k}{k}\binm{q+k}{k}}{\binm{2p+n+k}{k}\binm{2p-q+k}{k}}
 \frac{1+2p+2k}{1+2p+n+k}H_{p+k}\nnm\\
&&\xxqdn\xqdn\qdn\:\:=\:\frac{1}{2}\frac{\binm{2p+n}{n}\binm{p-q+n-1}{n}}{\binm{p+n}{n}\binm{2p-q+n}{n}}
\big\{H_p+H_{p+n}+H_{p-q-1}-H_{p-q+n-1}\big\}.
 \emn

 \begin{thm} \label{thm-f}
Let $x$ and $y$ be complex numbers. Then
 \bnm
 &&\xqdn\sum_{k=0}^{n}(-1)^k\binm{n}{k}
 \frac{\binm{\frac{x}{2}+k}{k}\binm{x-\frac{1}{2}+k}{k}\binm{y+k}{k}}
 {\binm{\frac{x-1}{2}+k}{k}\binm{x-\frac{1}{2}+n+k}{k}\binm{x-y-\frac{1}{2}+k}{k}}
  \frac{1+2x+4k}{1+2x+2n+2k}H_{2k}(x)\\
&&\xqdn\:\:=\:\frac{1}{2}\frac{\binm{x-\frac{1}{2}+n}{n}\binm{\frac{x-3}{2}-y+n}{n}}
{\binm{\frac{x-1}{2}+n}{n}\binm{x-y-\frac{1}{2}+n}{n}}
\big\{H_n(\tfrac{x-1}{2})-H_n(\tfrac{x-3}{2}-y)\big\}.
 \enm
\end{thm}

\begin{proof}
Substitute respectively $x$ and $z$ by $x-\frac{1}{2}$ and
$\frac{x}{2}$ in \eqref{equation-f} to achieve
 \bnm
 &&\xxqdn\sum_{k=0}^{n}(-1)^k\binm{n}{k}
 \frac{\binm{\frac{x}{2}+k}{k}\binm{x-\frac{1}{2}+k}{k}\binm{y+k}{k}}
 {\binm{\frac{x-1}{2}+k}{k}\binm{x-\frac{1}{2}+n+k}{k}\binm{x-y-\frac{1}{2}+k}{k}}
  \frac{1+2x+4k}{1+2x+2n+2k}\big\{H_{k}(\tfrac{x}{2})+H_{k}(\tfrac{x-1}{2})\big\}\\
&&\xxqdn\:\:=\:\frac{\binm{x-\frac{1}{2}+n}{n}\binm{\frac{x-3}{2}-y+n}{n}}
{\binm{\frac{x-1}{2}+n}{n}\binm{x-y-\frac{1}{2}+n}{n}}
\big\{H_n(\tfrac{x-1}{2})-H_n(\tfrac{x-3}{2}-y)\big\}.
 \enm
In terms of \eqref{relation}, the last equation can be manipulated
as Theorem \ref{thm-f} to finish the proof.
\end{proof}

Fixing $x=p$, $y=q$ with $p,q\in \mathbb{N}_0$ in Theorem
\ref{thm-f} and exploiting \eqref{equation-e}, we attain the
summation formula on harmonic numbers and generalized harmonic
numbers:
 \bmn\label{harmonic-ff}
 &&\xqdn\qqdn\sum_{k=0}^{n}(-1)^k\binm{n}{k}
 \frac{\binm{\frac{p}{2}+k}{k}\binm{p-\frac{1}{2}+k}{k}\binm{q+k}{k}}
 {\binm{\frac{p-1}{2}+k}{k}\binm{p-\frac{1}{2}+n+k}{k}\binm{p-q-\frac{1}{2}+k}{k}}
  \frac{1+2p+4k}{1+2p+2n+2k}H_{p+2k}
  \nnm\\
&&\xqdn\qqdn\:\:=\:\frac{1}{2}\frac{\binm{p-\frac{1}{2}+n}{n}\binm{\frac{p-3}{2}-q+n}{n}}
{\binm{\frac{p-1}{2}+n}{n}\binm{p-q-\frac{1}{2}+n}{n}}
\big\{2H_p+H_n(\tfrac{p-1}{2})-H_n(\tfrac{p-3}{2}-q)\big\}.
 \emn

\begin{thm} \label{thm-g}
Let $x$ be a complex number. Then
 \bnm
\qquad
\sum_{k=0}^{n}(-1)^k\binm{n}{k}\frac{\binm{2x+k}{k}}{\binm{2x+n+k}{k}}
 \frac{1+2x+2k}{1+2x+n+k}H_{k}^2(x)
=\frac{1}{2n}\frac{\binm{2x+n}{n}}{\binm{x+n}{n}^2}\big\{H_{n-1}-H_{n}(x)\big\}.
 \enm
\end{thm}

\begin{proof}
Theorem \ref{thm-e} can be written as
 \bnm
 &&\xqdn\qqdn\qdn\sum_{k=0}^{n}(-1)^k\binm{n}{k}\frac{\binm{2x+k}{k}\binm{2y+k}{k}}{\binm{2x+n+k}{k}\binm{2x-2y+k}{k}}
 \frac{1+2x+2k}{1+2x+n+k}H_{k}(x)\\
&&\xqdn\qqdn\qdn\:\:=\:\frac{1}{2(x-2y+n)}\frac{\binm{2x+n}{n}\binm{x-2y+n}{n}}{\binm{x+n}{n}\binm{2x-2y+n}{n}}
\bigg\{(x-2y)\big[H_n(x)-H_{n}(x-2y)\big]-\frac{n}{x-2y+n}\bigg\}.
 \enm
 Applying the derivative operator $\mathcal{D}_y$ to both sides of
 the last equation, we obtain
\bnm
 &&\xqdn\sum_{k=0}^{n}(-1)^k\binm{n}{k}\frac{\binm{2x+k}{k}\binm{2y+k}{k}}{\binm{2x+n+k}{k}\binm{2x-2y+k}{k}}
 \frac{1+2x+2k}{1+2x+n+k}H_{k}(x)\big\{H_k(2y)+H_k(2x-2y)\big\}\\
&&\xqdn\:\:=\:\frac{1}{2(x-2y+n)}\frac{\binm{2x+n}{n}\binm{x-2y+n}{n}}{\binm{x+n}{n}\binm{2x-2y+n}{n}}
\bigg\{(x-2y)A_n(x,y)-\frac{nB_n(x,y)}{x-2y+n}-\frac{2n}{(x-2y+n)^2}\bigg\},
 \enm
where the corresponding symbols stand for
 \bnm
&&\xqdn
A_n(x,y)=\big[H_n(x)-H_n(x-2y)\big]\big[H_n(2x-2y)-H_n(x-2y)\big]
-H_{n}^{\langle2\rangle}(x-2y),\\
&&\xqdn B_n(x,y)=H_n(x)+H_n(2x-2y)-2H_n(x-2y).
 \enm
Perform the replacement $y\to\frac{x}{2}$ in the last equation to
give Theorem \ref{thm-g}.
\end{proof}

Setting $x=p$ with $p\in \mathbb{N}_0$ in Theorem \ref{thm-g} and
utilizing \eqref{equation-e}, \eqref{harmonic-ee} and the relation
 \bnm
H_k^2(p)=\{H_{p+k}-H_p\}^2=H_{p+k}^2-2H_pH_{p+k}+H_p^2,
 \enm
 we get the summation formula on harmonic numbers:
 \bmn\label{harmonic-gg}
\:\sum_{k=0}^{n}(-1)^k\binm{n}{k}\frac{\binm{2p+k}{k}}{\binm{2p+n+k}{k}}
 \frac{1+2p+2k}{1+2p+n+k}H_{p+k}^2
=\frac{1}{2n}\frac{\binm{2p+n}{n}}{\binm{p+n}{n}^2}\big\{H_{n-1}-H_p-H_{p+n}\big\}.
 \emn

\begin{thm} \label{thm-h}
Let $x$ and $y$ be complex numbers. Then
 \bnm
 &&\xxqdn\sum_{k=0}^{n}(-1)^k\binm{n}{k}
 \frac{\binm{\frac{x}{2}+k}{k}^2\binm{x-\frac{1}{2}+k}{k}}
 {\binm{\frac{x-1}{2}+k}{k}^2\binm{x-\frac{1}{2}+n+k}{k}}
  \frac{1+2x+4k}{1+2x+2n+2k}H_{2k}^2(x)\\
&&\xxqdn\:\:=\:\frac{1}{4}\frac{\binm{x-\frac{1}{2}+n}{n}\binm{-\frac{3}{2}+n}{n}}
{\binm{\frac{x-1}{2}+n}{n}^2}
\big\{\big[H_n(\tfrac{x-1}{2})-H_n(-\tfrac{3}{2})\big]^2-H_{n}^{\langle2\rangle}(-\tfrac{3}{2})\big\}.
 \enm
\end{thm}

\begin{proof}
 Applying the derivative operator $\mathcal{D}_y$ to both sides of
 \eqref{equation-f}, we gain
 \bnm
  &&\xxqdn\xqdn\sum_{k=0}^{n}(-1)^k\binm{n}{k}\frac{\binm{x+k}{k}\binm{y+k}{k}\binm{z+k}{k}}
  {\binm{x+n+k}{k}\binm{x-y+k}{k}\binm{x-z+k}{k}}\frac{1+x+2k}{1+x+n+k}
 \\ &&\xxqdn\xqdn\:\,\times\:\:\big\{H_k(y)+H_k(x-y)\big\}\big\{H_k(z)+H_k(x-z)\big\}
 \\&&\xxqdn\xqdn\:\:=
 \frac{\binm{x+n}{n}\binm{x-y-z-1+n}{n}}{\binm{x-y+n}{n}\binm{x-z+n}{n}}
 \big\{C_n(x,y)-H_{n}^{\langle2\rangle}(x-y-z-1)\big\},
 \enm
where
$C_n(x,y)=\big[H_n(x-y)-H_n(x-y-z-1)\big]\big[H_n(x-z)-H_n(x-y-z-1)\big]$.
Employ the substitutions $x\to x-\frac{1}{2}$, $y\to\frac{x}{2}$,
$z\to\frac{x}{2}$ in the last equation to achieve
 \bnm
 &&\xxqdn\xqdn\sum_{k=0}^{n}(-1)^k\binm{n}{k}
 \frac{\binm{\frac{x}{2}+k}{k}^2\binm{x-\frac{1}{2}+k}{k}}
 {\binm{\frac{x-1}{2}+k}{k}^2\binm{x-\frac{1}{2}+n+k}{k}}
  \frac{1+2x+4k}{1+2x+2n+2k}\big\{H_{k}(\tfrac{x}{2})+H_{k}(\tfrac{x-1}{2})\big\}^2\\
&&\xxqdn\xqdn\:\:=\:\frac{\binm{x-\frac{1}{2}+n}{n}\binm{-\frac{3}{2}+n}{n}}
{\binm{\frac{x-1}{2}+n}{n}^2}
\big\{\big[H_n(\tfrac{x-1}{2})-H_n(-\tfrac{3}{2})\big]^2-H_{n}^{\langle2\rangle}(-\tfrac{3}{2})\big\}.
 \enm
According to \eqref{relation}, the last equation can be expressed as
Theorem \ref{thm-h} to complete the proof.
\end{proof}

Taking $x=p$ with $p\in \mathbb{N}_0$ in Theorem \ref{thm-h} and
using \eqref{equation-e}, \eqref{harmonic-ff} and the relation
 \bnm
H_{2k}^2(p)=\{H_{p+2k}-H_p\}^2=H_{p+2k}^2-2H_pH_{p+2k}+H_p^2,
 \enm
 we attain the summation formula on harmonic numbers and generalized harmonic numbers:
\bmn\label{harmonic-hh}
 &&\xxqdn\sum_{k=0}^{n}(-1)^k\binm{n}{k}
 \frac{\binm{\frac{p}{2}+k}{k}^2\binm{p-\frac{1}{2}+k}{k}}
 {\binm{\frac{p-1}{2}+k}{k}^2\binm{p-\frac{1}{2}+n+k}{k}}
  \frac{1+2p+4k}{1+2p+2n+2k}H_{p+2k}^2
  \nnm\\\nnm
&&\xxqdn\:\:=\:\frac{1}{4}\frac{\binm{p-\frac{1}{2}+n}{n}\binm{-\frac{3}{2}+n}{n}}
{\binm{\frac{p-1}{2}+n}{n}^2}
\big\{\big[H_n(\tfrac{p-1}{2})-H_n(-\tfrac{3}{2})\big]^2-H_{n}^{\langle2\rangle}(-\tfrac{3}{2})\\
&&\qquad\qquad\qquad\:\:\:
+\:4H_p\big[H_p+H_n(\tfrac{p-1}{2})-H_n(-\tfrac{3}{2})\big]\big\}.
 \emn
With the change of the parameters $p$ and $q$, \eqref{harmonic-aa},
\eqref{harmonic-bb}, \eqref{harmonic-cc}, \eqref{harmonic-dd} and
 \eqref{harmonic-ee}-\eqref{harmonic-hh}
 can create a lot of concrete harmonic number identities. The
 corresponding results will not be displayed here.

 \textbf{Acknowledgments}

The work is supported by the National Natural Science Foundation of
China (No. 11301120).



\end{document}